\documentclass[11pt]{amsart}
\usepackage{graphicx, pdfsync}
\usepackage[colorlinks=true, citecolor=red]{hyperref}
\usepackage{amsmath,amssymb,amsfonts,MnSymbol, mathrsfs}

\usepackage{algorithm}
\usepackage{xcolor}
\usepackage{listings}
\usepackage[colorlinks=true, citecolor=red]{hyperref}

\usepackage{tikz,pgfplots,bm}
\usetikzlibrary{patterns}

\numberwithin{equation}{section}

\newtheorem{theorem}{Theorem}
\newtheorem{proposition}[theorem]{Proposition}%
\newtheorem{remark}{Remark}%
\newtheorem{notation}{Notation}

\newtheorem{definition}{Definition}%

\newtheorem{lemma}{Lemma}

\newtheorem{corollary}{Corollary}

\def\e{{\rm e}}
\def\ii{{\rm i}}

\def\vS{{\mathbf{S}}}
\def\vI{{\mathbf{I}}}
\def\vJ{{\mathbf{J}}}

\def\vF{{\mathbf{F}}}

\def\E{{\mathbb{E}}}
\def\R{{\mathbb{R}}}
\def\S{{\mathbb{S}}}
\def\N{{\mathbb{N}}}

\def\Z{{\mathbb{Z}}}

\newcommand{\ds}{\displaystyle}

\def\p{\partial}
\def\n{\nabla}

\def\d{\hbox{d}}
\def\One{{\bf 1}}

\def\eq#1{\begin{equation}#1\end{equation}}
\def\eeq#1{\begin{eqnarray}#1\end{eqnarray}}
\def\eeqn#1{\begin{eqnarray*}#1\end{eqnarray*}}

\def\vx{{\bf x}}

\def\vom{{\bm\omega}}
\def\SS{{\mathbb{S}_2}}

\title[Polarized Light in a Graded Index Atmosphere]{ Numerical Simulation of Polarized Light and Temperature in a Stratified Atmosphere with a Slowly Varying Refractive Index}

\author[O. Pironneau]{Olivier Pironneau}

\address[O. Pironneau]{ LJLL, Sorbonne Universit\'e, Place Jussieu, Paris, France.}

\email{olivier.pironneau@sorbonne-universite.fr }
\date{ 15 December 2024: Dedicated to Roger Temam for his 80$^{th}$ birthday}

\begin{document}
\begin{abstract}
This article is an attempt to elucidate the effect of a slowly varying refractive index on the temperature in a stratified atmosphere, with a particular focus on  greenhouse gases such as \texttt{CO}$_2$. It validates an iterative method for the vector radiative transfer equations (VVRTE) called Iterations on the Source. As the system proposed by Chandrasekhar and Pomraning is not well posed for all rays directions when the refractive index varies, so instead we solve an integral representation of VRTE without the problematic rays.
A mathematical proof is given showing monotonicity, convergence of the iterations and existence and uniqueness. Furthermore the convergence is geometric if the absorption is not too large. Some numerical tests are performed showing the effect of a layer of cloud with a refractive index greater than air, polarisation and scattering. 
\end{abstract}

%
\maketitle

\section*{Introduction}

The investigation of polarized light in a stratified or plane parallel atmosphere provides crucial insights into the dynamics of light propagation and polarization within complex atmospheric environments; it  also holds significant implications for understanding the impact of greenhouse gases on the Earth's climate system in the presence of slightly polarizing clouds \cite{measureN}. Building upon the Variable refractive index Vector Radiative Transport Equation (VVRTE), as established by the seminal works of Chandrasekhar \cite{CHA} and Pomraning \cite{POM}, this article aims to elucidate numerically the effect of the refractive index on the temperature  with a particular focus on the greenhouse gas  (GHG) such as \texttt{CO}$_2$.

Varying concentrations of greenhouse gases  are modelled by their effects on absorption and scattering coefficients in given ranges of frequencies. The study shows that the numerical method is capable of reproducing small changes due to GHG and the influence of refraction from clouds.

Clouds have a refractive index close to air  and varying smoothly, unlike an air/water interface for which the Fresnel laws need to be applied (see   \cite{zhao}, \cite{garcia} and  \cite{fresnel} for a numerical implementation).

The computer graphics community has used VVRTE with varying refractive indices for realistic rendering (see \cite{Liu}, \cite{graphics} and the references therein).  But most of their numerical implementations assume a given temperature field.

Temperature variations are at the core of the present study, but the problem had also to be reformulated because the PDE (partial differential equations) of VVRTE are not well posed, due to total refraction of some rays. 
Hence the PDEs are converted into a set of integral equations for. the admissible rays using the method of characteristics. Iterations on the Source are then used together with Newton iterations on the temperature equations. Convergence and monotony are analysed as in \cite{OP2023}.

The numerical implementation is more difficult because exponential integrals cannot be used, so a precision study was added.

Note that this approach can be generalised in 3D (non stratified atmospheres) as in \cite{JCP}, \cite{JCP2}.

\section{Fundamental Equations}

In classical physics, light in a medium $\Omega$  is an electromagnetic radiation satisfying Maxwell's equations.
The electric field ${\bf E}$ of a monochromatic plane wave of frequency $\nu$ propagating in direction ${\bf k}$, 
${\bf E}={\bf E_0}\exp(\ii({\bf k\cdot\vx} - \nu t))$,  is a solution to Maxwell equations which is suitable to describe the propagation of a ray of light in a uniform medium when $\nu$ is very large.

Such radiations are characterized by their Stokes vector $\vI$,  made of  the irradiance $I$ and 3 functions $Q,U,V$ to define the state of polarization. The  radiation source $\vF$  for an unpolarized-emitting black-body  is given by the  Planck function $B_\nu(T)=\nu^3(\e^\frac\nu T - 1)^{-1}$. The scales used in this article are defined in \cite{FGOP3}. If ${\bf E_0}$ is written in polar coordinates,
\begin{equation*}
 \vI:= \left[\begin{matrix}
I\cr Q\cr U \cr V 
\end{matrix}\right]
=
 \frac12\sqrt{\frac{\epsilon_0}{\mu_0}}\left[\begin{matrix}
{E_0}_\theta {E_0}_\theta^* + {E_0}_\phi {E_0}_\phi^*\cr 
{E_0}_\theta {E_0}_\theta^* - {E_0}_\phi {E_0}_\phi^*\cr 
{E_0}_\theta {E_0}_\phi^* - {E_0}_\phi {E_0}_\theta^* \cr 
{E_0}_\phi {E_0}_\theta^* - {E_0}_\theta {E_0}_\phi^* 
\end{matrix}\right]
\qquad\vF := \left[\begin{matrix}
\kappa_a B_\nu(T)\cr 0\cr 0 \cr 0 
\end{matrix}\right],
\end{equation*}
where $\epsilon_0$ and $\mu_0$ are the electric permittivity and magnetic permeability of the medium.
The parameter $\kappa_a$ is related to absorption and scattering (see \eqref{def} below)   
which, by the way, are quantum effects, not described by Maxwell's equations, except as an analogy to the limit in a random media with vanishing variations \cite{papanico}. Using  $\vI$ rather than ${\bf E}_0$ one models absorption and scattering by a PDE system known as VVRTE \cite{POM} p152, \cite{zhao},  
\begin{align}\label{fundamental}
\tilde\vI:=\frac{\vI}{n^2}:\quad \frac nc\partial_t\tilde\vI &+ \vom\n_\vx \tilde\vI + \frac{\n_\vx n}{n}\cdot\n_\vom\tilde\vI+\kappa\tilde\vI +\N\vI-\vI\N
\cr &
= \int_\SS\Z(\vx,\vom':\vom)\tilde\vI\d\omega' + \tilde{\vF},
\end{align}
for all $\vx\in\Omega,\vom\in\SS$, where $c$ is the speed of light, $n$ the refractive index of the medium ,  $\SS$ the unit sphere, $\kappa$ the absorption  and $\Z$ is the phase scattering matrix for rays $\vom'$ scattered in direction $\vom$.  $\N$ is a tensor which is related to the torsion of the characteristics of the rays \cite{zhao} and which is zero for plane parallel problems.  Finally, $\vF$ is the volumic source term due to the black-body radiations of air at temperature $T$.  
It is assumed that  $n$ depends weakly only, on position $\vx\in\Omega$; $\kappa$ depends  smoothly on $\vx$  and strongly on $\nu$. Because $c$ is very large, the term $\frac1c\partial_t\vI$ is neglected. 
Thermal equilibrium is assumed because convection by wind and molecular thermal diffusion are very small so the radiation heat source in the temperature equation of the gas is zero:
\begin{equation}
\label{thermal}
\n_\vx\cdot\int_{\R_+}\int_{\S_2} I\vom\d\omega\d\nu=0 ~~\hbox{ at all points $\vx\in\Omega$}.
\end{equation}
As we shall see below (Proposition \ref{prop1}), this  scalar equation defines the local temperature $\vx\mapsto T(\vx)$.
\begin{notation}
~

\begin{itemize}
\item On all variables, the tilde indicates a division by $n^2$. 
\item
To improve readability, arguments of functions are sometimes written as indices like $\kappa_\nu$ and $n_z$.
\end{itemize}
\end{notation}
Following \cite{Liu}, given a cartesian frame ${\bf i},{\bf j},{\bf k}$,  the third term on the left in \eqref{fundamental} is computed in polar coordinates, with 
\[
\vom := {\bf i}\sin\theta\cos\varphi+{\bf j}\sin\theta\sin\varphi+{\bf k}\cos\theta,
\quad{\bf s}_1:=-{\bf i}\sin\varphi+{\bf j}\cos\varphi,
\]
\begin{align*}
\n_\vx\log{n}\cdot\n_\vom\tilde\vI = \frac{1}{\sin \theta} \frac{\partial}{\partial \theta}\left\{\tilde\vI(\cos\theta \vom-\boldsymbol{k}) \cdot \nabla_\vx\log n\right\} 
+\frac{1}{\sin \theta} \frac{\partial}{\partial \varphi}\left\{\tilde\vI~\boldsymbol{s}_1 \cdot \nabla_\vx\log n\right\}.
\end{align*}
When $n$ does not depend on $x,y$ but only on $z$, it simplifies to
\begin{equation*}
\n_\vx\log{n}\cdot\n_\vom\tilde\vI = (\partial_z\log n)\cdot \partial_\mu\left\{(1-\mu^2)\tilde\vI\right\}
\quad\text{ where }\mu=\cos\theta.
\end{equation*}

\section{The Stratified Case}

For an atmosphere of thickness $Z$ over a flat ground, the domain is $\Omega=\R^2\times\times(0,Z)$, and all variables are independent of $x,y$.
In \cite{CHA}, p40-53, expressions for the phase matrix $\Z$ are given for Rayleigh and isotropic scattering for the  $\varphi$-averaged of $I$ and $Q$,
\begin{align}\label{scatmat}&\ds
\bar I:=\frac1{2\pi}\int_0^{2\pi}I\d\varphi,
\quad
\bar Q:=\frac1{2\pi}\int_0^{2\pi}Q\d\varphi,
\cr &
 \bar\Z_R = \frac32\left[\begin{matrix}
2(1-\mu^2)(1-\mu'^2)+\mu^2\mu'^2 & \mu^2 \cr
\mu'^2 & 1\cr
\end{matrix}\right].
\quad
\bar\Z_I = 
\frac12\left[\begin{matrix}
1&1 \cr
1&1\cr
\end{matrix}\right].
\end{align}
The general expression of the phase matrix for Rayleigh scattering according to S. Chandrasekhar \cite{CHA} is given in Appendix \ref{appendix}.

For a given $\beta\in[0,1]$, we shall consider a combination of $\beta\Z_R$  (Rayleigh scattering) plus $(1-\beta)\Z_I$ (isotropic scatterings) \cite{CHA},\cite{POM},\cite{POM2}. As it is understood that no variable depends on $\varphi$, we drop the overline. 

The two other components of the Stokes vectors have autonomous equations,
\begin{align}\label{U}
\mu\p_z \tilde U + \partial_z\log n (1-\mu^2)\partial_\mu\tilde U
+\kappa\tilde U = 0, 
\\ \label{V}
\mu\p_z \tilde V + \partial_z\log n (1-\mu^2)\partial_\mu\tilde U
+\kappa\tilde V = \frac\mu{2}\int_{-1}^1 \mu'\tilde V(z,\mu')\d\mu'.
\end{align}
\begin{notation}
Denote the scattering coefficient $a_s\in[0,1)$, which, as $\kappa$ (the absorption), is a function of altitude $z$ and frequency $\nu$. Define  
\eq{\label{def}
\kappa_s =\kappa a_s, \qquad \kappa_a=\kappa-\kappa_s = \kappa (1-a_s).
}
\end{notation}
From \eqref{U},\eqref{V} we see that, if the light source at the boundary is unpolarized then  $U=V=0$.

 The system for $\tilde I$ and $\tilde Q$ is derived from \eqref{fundamental} and \eqref{scatmat} ,
\begin{equation}\label{lq}\ds 
\left\{
\begin{aligned}&
\mu \p_z\tilde I + \partial_z\log n (1-\mu^2)\partial_\mu\tilde I+ \kappa\tilde I 
\cr&
\hskip 1cm=\kappa_a \tilde B_\nu + \frac{\kappa_s}2\int_{-1}^1 \tilde I\d\mu'
+ \frac{\beta\kappa_s}4 P_2(\mu)\int_{-1}^1 [P_2\tilde I-(1-P_2 )\tilde Q]\d\mu',
\cr&
\mu \p_z \tilde Q + \partial_z\log n (1-\mu^2)\partial_\mu\tilde Q + \kappa\tilde Q 
\cr&
\hskip 1cm= -\frac{\beta\kappa_s}4 (1-P_2(\mu))\int_{-1}^1 [P_2\tilde I-(1-P_2 )\tilde Q]\d\mu',
\end{aligned}
\right.
\end{equation}
where  $P_2(\mu)=\tfrac12(3\mu^2-1)$.
The temperature $T(z)$ is linked to $I$ by \eqref{thermal} which, in the case of \eqref{lq} is as follows. 

\begin{proposition}\label{prop1}
Thermal equilibrium for \eqref{lq} (or \eqref{lllreq} below)  is
\eeq{\label{fort}
   \int_{\R_+}\kappa_a\big[\tilde B_\nu(T)
 - \tfrac12\int_{-1}^1 \tilde I\d\mu\big]\d\nu=0 ~\hbox{ for all }z\in[0,Z].
}
\end{proposition}
\begin{proof}
Averaging in $\mu$ the first equation of \eqref{lq} leads to
\begin{align*}&
0=\frac1{n^2}\nabla_\vx\cdot\int_\SS \vom  I = \frac1{n^2}\frac12\int_{-1}^1\mu \partial_z  I\d\mu 
=\frac12\int_{-1}^1\mu\partial_z \tilde I\d\mu + \partial_z\log n \int_{-1}^1\mu \tilde I\d\mu
\cr&
= \partial_z\log n \int_{-1}^1[\mu \tilde I-\frac12(1-\mu^2)\partial_\mu\tilde I]\d\mu
-\frac12\int_{-1}^1[ \kappa\tilde I-\kappa_a \tilde B_\nu - \frac{\kappa_s}2\int_{-1}^1\tilde I\d\mu']\d\mu,
\end{align*}
because $\int_{-1}^1P_2(\mu)\d\mu=0$. Now the first term on the right integrates to zero and $\kappa-\kappa_s=\kappa_a$.
\end{proof}
The following proposition is straightforward,
\begin{remark}
The light can be described either by $\tilde I$ and $\tilde Q$ and \eqref{lq} or two orthogonal components $I_l,I_r$, such that  $I=I_l+I_r$ and $Q=I_l-I_r$, then:
 \begin{equation}\label{lllreq}\ds 
\left\{
\begin{aligned}&
\mu\p_z  {\tilde I}_l + \partial_z\log n (1-\mu^2)\partial_\mu\tilde I_l +\kappa {\tilde I}_l 
 \cr&
 \hskip 2cm= \frac{3\beta\kappa_s}8 \int_{-1}^1([2(1-\mu'^2)(1-\mu^2)+\mu'^2\mu^2]{\tilde I}_l + \mu^2 {\tilde I}_r)\d\mu'
\cr&
\hskip 2cm + \frac{(1-\beta)\kappa_s}4\int_{-1}^1[{\tilde I}_l+ {\tilde I}_r]\d\mu'  +  \frac{\kappa_a}2 \tilde B_\nu(T(z)), 
 \cr&
 \mu\p_z  {\tilde I}_r + \partial_z\log n (1-\mu^2)\partial_\mu\tilde I_r+\kappa {\tilde I}_r =\frac{3\beta\kappa_s}8 \int_{-1}^1(\mu'^2 {\tilde I}_l + {\tilde I}_r)\d\mu' 
 \cr&
 \hskip 2cm  
+ \frac{(1-\beta)\kappa_s}4\int_{-1}^1[{\tilde I}_l+ {\tilde I}_r]\d\mu'  +  \frac{\kappa_a}2 \tilde B_\nu(T(z)). 
\end{aligned}
\right.
\end{equation}
\end{remark}
\subsection{Iterations on the Source}
Consider the following system $[I^{m}_l,I^{m}_r,T^{m}]\mapsto [I^{m+1}_l,I^{m+1}_r,T^{m+1}]$,
\eeq{\label{lllr}&&
 \mu\p_z  \tilde I_l^{m+1} + \partial_z\log n  (1-\mu^2)\partial_\mu\tilde I_l^{m+1}+\kappa \tilde I_l^{m+1} 
 \cr&&
\hskip 1.cm= \frac{3\beta\kappa_s}8 \int_{-1}^1([2(1-\mu'^2)(1-\mu^2)+\mu'^2\mu^2]\tilde I_l^m + \mu^2 \tilde I_r^m)\d\mu'
\cr&&
\hskip 3cm+ \frac{(1-\beta)\kappa_s}4\int_{-1}^1[\tilde I_l^m+ \tilde I_r^m]\d\mu'  +  \frac{\kappa_a}2 \tilde B_\nu(T^m)
 \cr&&
 \cr&&
 \mu\p_z  \tilde I_r^{m+1} + \partial_z\log n  (1-\mu^2)\partial_\mu\tilde I_r^{m+1}+\kappa \tilde I_r^{m+1} 
 \cr&&
\hskip 0.5cm= \frac{3\beta\kappa_s}8 \int_{-1}^1(\mu'^2 \tilde I_l^m +  \tilde I_r^m)\d\mu' 
 + \frac{(1-\beta)\kappa_s}4\int_{-1}^1[\tilde I_l^m+ \tilde I_r^m]\d\mu'  +  \frac{\kappa_a}2 \tilde B_\nu(T^m)
\cr&&
\cr&&\hskip 0.5cm 
\int_{\R_+}\kappa_a \tilde B_\nu(T^{m+1})\d\nu = \int_{\R_+}\kappa_a \left(\tfrac12\int_{-1}^1(\tilde I_l^{m}+\tilde I_r^{m})\d\mu\right) \d\nu~~\forall z.
}
\section{The Method of Characteristics}
By analogy with the non-refractive case we write the transport equations above as a 2-system for $\tilde\vI:=[\tilde I_l^{m+1},\tilde I_r^{m+1}]^T$,
\begin{align}\label{funda}&
\partial_z\tilde\vI 
+\frac{1-\mu^2}{\mu}\partial_z\log n \partial_\mu\tilde\vI + \frac\kappa\mu\tilde\vI = \frac1\mu\tilde\vS(\mu,z),
\end{align}
where $\tilde\vS=\vS_0 + \mu^2\tilde\vS_2$ and the $\vS_k$ are linear combinations of  $\int_{-1}^1\mu'^q\vI^m\d\mu',~q=0,2$.
The characteristic curves are given in terms of 2 constants $z_0,\mu_0$, by
\eq{\label{zmu}
\dot z= 1,\quad \dot \mu\mu=(1-\mu^2)\partial_z\log n
\Rightarrow
z(s)=s+z_0,\quad \mu^2(s)=1-(1-\mu^2_0)\frac{n_0^2}{n^2_{z(s)}}.
}
Then, in terms of $\bar\vI(s)=\tilde\vI(z(s),\mu(s))$, \eqref{funda} is
\begin{equation}\label{funda2}
\frac{\d\bar\vI}{\d s} + \frac{\kappa(z(s))}{\mu(s)}\bar\vI = \frac1{\mu(s)}{\tilde\vS(\mu(s),z(s))}.
\end{equation}
Compatible boundary conditions are : $\tilde\vI(z_0,\mu_0)$ given at all $z_0\in\partial\Omega$ and all $\mu_0$ where the rays enter the domain $\Omega$.  For instance,
\begin{definition}
Assume that $\tilde\vI$ is given at $\mu>0,z_0=0$ and $\mu<0,z=Z$: 
\begin{equation}\label{h3}
\tilde\vI(0,\mu)=\tilde\vI_0(\mu)\ge 0 \hbox{ for all }\mu>0,
\quad 
\tilde\vI(Z,\mu)=0 \hbox{ for all }\mu<0.
\end{equation}
\end{definition}
Let $z(s),\mu(s)$ be given by \eqref{zmu}. Denote $\kappa(s)=\kappa(z(s))$. Then the solution
of \eqref{funda2} is
\begin{align}\label{gensol}
\tilde\vI(z(s),\mu(s)) &= \One_{\mu_0>0}\left[\e^{-\int_0^s\frac{\kappa(s')}{\mu(s')}d s'}\tilde\vI_0(\mu_0) 
+ \int_0^s\frac{\e^{-\int^s_{s'} \frac{\kappa(s'')}{\mu(s'')}d s''}}{\mu(s')}{\tilde\vS(s')}\d s'\right]
\cr&
-\One_{\mu_0<0}\left[\int_s^S \frac{\e^{\int_{s}^{s'}\frac{\kappa(s'')}{\mu(s'')} d s'' }}{\mu(s')}{\tilde\vS(s')}\d s'\right].
\end{align}
where $S$ is such that $(z(S),\mu(S))$ is the exit point of the characteristic.

\subsection{Compatible Characteristics}
Obviously \eqref{gensol} holds only if there is an exit point $z(S)=Z$. In other words for every altitude $\bar z$ and direction $\bar\mu$ there must exist a characteristic $\{z(s),\mu(s)\}_0^S$ and a $\mu_0>0$  such that for some $\bar s$, with  $\{z(\bar s)=\bar z,~\mu(\bar s)=\bar\mu\}$ and $z(0)=0,\mu(0)=\mu_0\ge 0$, $z(S)=Z, \mu(S)=\mu_Z\le 0$. Otherwise the problem is ill-posed. By \eqref{zmu}, with $z_0=0$, all directions $\mu$, with $\mu^2(z)<1-\frac{n^2_0}{n^2_{z}}$, are forbidden  at $z$ because the equation for $\mu_0$ is 
\[
\mu_0^2=1-(1-\mu^2(z))\frac{n^2(z)}{n_0^2}< 1-\frac{n^2(z)}{n_0^2} + (1-\frac{n^2_0}{n^2_{z}})\frac{n^2(z)}{n_0^2}<0.
\]
By the same argument made at $Z$ instead of $0$, we must have $\mu^2(z)\ge 1-\frac{n^2_Z}{n^2_{z}}$ for $\mu_Z$ to exist.
\begin{definition}
A ray direction $\theta$ is admissible at $z$ if
\begin{equation}\label{adimissible}
\mu^2(z):=\cos^2\theta \ge \max\{0,1-\frac{n^2_0}{n^2_{z}},1-\frac{n^2_Z}{n^2_{z}}\}.
\end{equation}
\end{definition}
Note that
\begin{align}\label{admin}
 \mu^2(z)&>1-\frac{n^2_{z''}}{n^2_{z}},\quad
 \mu^2(z')=1-(1-\mu^2(z))\frac{n^2_{z}}{n^2_{z'}} \implies
 \cr \mu^2(z')&> 1-\frac{n^2_{z}}{n^2_{z'}} + (1-\frac{n^2_{z''}}{n^2_{z}})\frac{n^2_{z}}{n^2_{z'}}
 >1-\frac{n^2_{z''}}{n^2_{z'}}.
\end{align}
Consequently, applying it at $z''=0$ and $Z$ shows that if $\mu(z)$ is admissible at $z$, then $\mu(s)$ is admissible at all $s\in[0,Z]$. Note also that, in \eqref{gensol}, $\One_{\mu_0}$ can be replaced by $\One_{\mu(s)}$.
\begin{corollary}
Assume \eqref{h3}.  Then  \eqref{funda} is well posed for all $z\in(0,Z)$ and all $\mu$ admissible at $z$ and it is given by
\begin{align}\label{gensol2}
\tilde\vI(z,\mu) &= \One_{\mu>0}\left[\e^{-\int_0^z\frac{\kappa(z')}{\mu(z')}d z'}\tilde\vI_0(\mu_0) 
+ \int_0^z\frac{\e^{-\int^z_{z'} \frac{\kappa(z'')}{\mu(z'')}d z''}}{\mu(z')}{\tilde\vS(z',\mu(z'))}\d s'\right]
\cr&
+\One_{\mu<0}\left[\int_z^Z \frac{\e^{-\int_{z}^{z'}\frac{\kappa(z'')}{\mu(z'')} d z'' }}{\mu(z')}{\tilde\vS(z',\mu(z'))}\d z'\right],
\end{align}
where $\mu(s)=\sqrt{1-(1-\mu^2)\frac{n_z^2}{n^2_{s}}}$, $s=z',z''$.
\end{corollary}
For Earth's atmosphere subject to unpolarized infrared radiation from the ground of intensity $c_E$ at temperature $T_E$, $\tilde\vI_0=\tilde\vI'_0\mu(0)$ with  $\tilde\vI'_0=c_E \tilde B_\nu(T_E)[\tfrac12,\tfrac12]^T$.
Now, 
\begin{align}\label{J0}
\tilde\vJ_0(z)&:=\tfrac12 \int_{\mu\in(-1,\mu_*)\cup(\mu^*,1)}\tilde\vI(z,\mu)\d\mu 
=
\tfrac12\tilde\vI'_0\int_{\mu^*}^1\varphi(\kappa,0,z,\mu)\mu_0\d\mu
\cr&
+  \tfrac12 \int_0^z\left(\int_{\mu^*}^1 \varphi(\kappa,z',z,\mu)\frac1{\mu_{z'}}\d\mu\right){\tilde\vS_0(z')}\d z'
\cr&
+  \tfrac12 \int_z^Z\left(\int_{-1}^{\mu_*} \varphi(\kappa,z',z,\mu)\frac1{\mu_{z'}}\d\mu\right){\tilde\vS_0(z')}\d z'
\cr&
+  \tfrac12 \int_0^z\left(\int_{\mu^*}^1 \varphi(\kappa,z',z,\mu){\mu_{z'}}\d\mu\right)\tilde\vS_2(z')\d z'
\cr&
+  \tfrac12 \int_z^Z\left(\int_{-1}^{\mu_*} \varphi(\kappa,z',z,\mu){\mu_{z'}}\d\mu\right)\tilde\vS_2(z')\d z',
\cr&
=\tfrac12\tilde\vI'_0\int_{\mu^*}^1\varphi(\kappa,0,z,\mu)\mu_0\d\mu
\cr&
+   \tfrac12 \int_0^Z\left(\int_{\mu^*}^1 \varphi(\kappa,z',z,\mu)\frac1{\mu_{z'}}\d\mu\right){\tilde\vS_0(z')}\d z'
\cr&
+  \tfrac12 \int_0^Z\left(\int_{\mu^*}^1 \varphi(\kappa,z',z,\mu){\mu_{z'}}\d\mu\right)\tilde\vS_2(z')\d z',
\end{align}
where the last equality holds only if $\mu_*=-\mu^*$, i.e. $n_0=n_Z$, and where
\begin{align*}&
\varphi(\kappa,z',z,\mu):= \exp(-\left|\int_z^{z'} \frac{\kappa(z'')}{\eta(\mu,z,z'')}\d z''\right|)
,~
\eta(\mu,z,z''):=(1-(1-\mu^2)\frac{n^2_{z}}{n^2_{z''}})^\frac12,
\cr&
\mu_0= \eta(\mu,z,0),~ \mu_{z'}=\eta(\mu,z,z'),~ \mu_{z''}=\eta( \mu,z,z'').
\end{align*}
\begin{remark}
Neglecting rays with $\mu\in(\mu_*,\mu^*)$ means that $\tilde \vJ_0$ is not the full average light intensity and that some of the boundary conditions in \eqref{h3} are not satisfied.  On the other hand now the problem has a well defined meaning.
\end{remark}
\begin{lemma}
\begin{equation}\label{mubar}
\varphi(\kappa,z',z,\mu)\le \e^{-\frac1\mu{\int_z^{z'}\frac{\kappa(z'')}{1+\varepsilon(z'',z)}d z''}},
\text{ where $\varepsilon(s,z) =\max\{1,\frac{n_s}{n_z}\}-1$}.
\end{equation}
\end{lemma}
\begin{proof}
If $n_{z''}\le n_z$ then $\eta^2(\mu,z,z'')\le 1-(1-\mu^2)=\mu^2$. So 
\[
\varphi(\kappa,z',z,\mu)\le \e^{-\frac1\mu\int_z^{z'}\kappa(z'')d z''}
\le \e^{-\frac1\mu{\int_z^{z'}\frac{\kappa(z'')}{1+\varepsilon}d z''}},~\forall \varepsilon\ge 0.
\]
If $n_{z''}>n_z$ then $\varepsilon(z'',z)>0$ and $\eta^2(\mu,z,z'')\le \mu^2(1+\varepsilon(z'',z))^2$.
\end{proof}

By analogy with exponential integrals $E_k$, let us define 
\begin{align*}
\E_k(\kappa,z,z')&:=\int_{\mu^*}^1\varphi(\kappa,z',z,\mu){\mu^{k-2}(z')}\d\mu
\end{align*}
Then
\begin{align} \label{Jdef}
\tilde\vJ_0(z)
&=
\frac{c_E}2 \tilde B_\nu(T_E) \E_3(\kappa,z,0)
+  \sum_{k=1,3}\tfrac12\int_0^Z\E_k(\kappa,z,z'){\tilde\vS_{k-1}(z')}\d z'.~~
\end{align}

\section{Convergence of the Iterations on the Source}
\subsection{Boundedness and Geometric Convergence}
%
Denote $\tilde I^{m+1}=\tilde I_l^{m+1}+\tilde I_r^{m+1}$ and let
\[
\tilde J^{m+1}_0(z):=\tfrac12\int_{-1}^1 \tilde I^{m+1}\d\mu,.
\]
Add the first 2 equations of \eqref{lllr} and notice that
$\mu'^2+2(1-\mu'^2)(1-\mu^2)+\mu'^2\mu^2\le 2$ so that $\forall\mu,z,\nu$,
\eeq{\label{ineqI}&&
 \mu\p_z  \tilde I^{m+1} + \partial_z\log n  (1-\mu^2)\partial_\mu\tilde I^{m+1}+\kappa \tilde I^{m+1} =:\S^m \le \kappa_a \tilde B^m + (\beta+1)\kappa_s \tilde J_0^m,
  \cr&&
  \hskip1cm\int_{\R_+}\kappa_a \tilde B(T^{m+1})\d\nu = \int_{\R_+}\kappa_a \tilde J_0^{m+1}\d\nu. 
}
Therefore, 
\eeq{\label{JJ}
J_0^{m+1}(z)\le \frac{S_E}2 \E_3(\kappa, z,0) \tilde B(T_E) + \tfrac12\int_0^Z  \E_1(\kappa,z,z')(\kappa_a \tilde B^m + (\beta+1)\kappa_s \tilde J_0^m)\d z'.
\cr&&}
By \eqref{mubar},
\begin{align}\label{falseE}
\E_k(\kappa,z,z')&\le (1+\varepsilon(z,z'))^{(k-2)^+}\int_0^1\e^{-\frac1\mu\int_z^{z'}\frac{\kappa(z'')}{1+\varepsilon(z,z'')}d z''}\mu^{k-2}\d\mu
\cr&
=(1+\varepsilon(z,z'))^{(k-2)^+}E_k(\int_z^{z'}\frac{\kappa(z'')}{1+\varepsilon(z,z'')}\d z'')
\cr&
\le
(1+\varepsilon_M)^{(k-2)^+}E_k(\kappa_\epsilon(z-z')),
\end{align}
because $x\mapsto E_k(x)$ is decreasing and because we have set $\kappa_\epsilon=\kappa_m/(1+\varepsilon_M)$ with
$\kappa_m=\min_{z}\kappa(z)$, $\varepsilon_M=\max_{z,z'}\varepsilon(z,z')$. 

Denote $C_1(X)=1-E_0(X)$ and observe that
\[
\int_0^X E_1(x)\d x=\int_1^\infty\int_0^X \frac{\e^{-x t}}{t}\d x\d t = 
\int_1^\infty \frac{1-\e^{-X t}}{t^2}\d t = 1-E_0(X) < 1,
\] 
and that
\begin{align}
\kappa_\epsilon\int_0^Z E_1(\kappa_\epsilon|z'-z|)\d z' 
\le \int_0^{\kappa_\epsilon Z} E_1(|s-\kappa_\epsilon z|)\d s
\cr
\le
\int_0^{\kappa_\epsilon z} E_1(\theta)\d\theta + \int_0^{\kappa_\epsilon(Z-z)} E_1(\theta)\d\theta\le 2C_1(Z\kappa_\epsilon).
\end{align}

%
Multiply \eqref{JJ} by $\kappa_a$ and integrate in  $\nu$,
\eeqn{&&\ds
\int_{\R_+}\kappa_a \tilde J_0^{m+1}\d\nu
\le 
\frac{S_E}2 (1+\varepsilon_M)\int_{\R_+}E_3(\kappa_\epsilon z)\kappa_a \tilde B(T_E)\d\nu
\cr&&
+\int_{\R_+}  \kappa_a\left(\int_0^Z \tfrac12 E_1(\kappa_\epsilon|z'-z])(\kappa_a \tilde B^m 
+ (\beta+1)\kappa_s \tilde J_0^m)\d z'\right)\d\nu.
}
Then, as $E_3(y)\leq E_3(0)\le \frac12 $, and by Stefan's law $\int_{\R_+}B_\nu(T)\d\nu=\frac{(\pi T)^4}{15}$,
\eeq{\label{HH}&\ds
H^{m+1}(z):=&\int_{\R_+}\kappa_a \tilde J_0^{m+1}\d\nu 
\le 
\frac{S_E}4 (1+\varepsilon_M)\kappa_M(1-a_m)\frac{\pi^4}{15}T_E^4
\cr&&
+(1-a_m) \frac{\kappa_M}{\kappa_\epsilon}C_1(Z\kappa_\epsilon)\max_z\int_{\R_+}(\kappa_a \tilde B^m + (\beta +1)\kappa_s \tilde J_0^m)\d\nu,
}
where  $a_m=\inf_{\nu,z} a_s$. 
Finally, by using the last equation of \eqref{ineqI}, we can replace $\kappa_a \tilde B^m$ by $\kappa_a \tilde J_0^m$ and as $\kappa_a+\kappa_s=\kappa$, it shows that
\[
H^{m+1}(z) \le  R +(1+\varepsilon_M)\frac{\kappa^2_M}{\kappa_m} C_1(\frac{Z\kappa_m}{1+\varepsilon_M})\frac{1-a_m}{1-a_M}( 1+\beta_M a_M)  \max_z H^m(z),
\]
with $R:=\frac{S_E}4 \kappa_M(1-a_m)\frac{\pi^4}{15}T_E^4.$
It implies that $\max_z H^m(z)$ is bounded by $RZ/(1-\eta) + \eta^m \max_z H^0$ if
\eeq{\label{bound}
 \eta:= (1+\varepsilon_M)\frac{\kappa^2_M}{\kappa_m} C_1(\frac{Z\kappa_m}{1+\varepsilon_M})\frac{1-a_m}{1-a_M}( 1+\beta_M a_M) < 1,
 }
and, in turn, $\{T^m\}_m$ is bounded because $\kappa_m(1-a_M)\frac{(pi T(z))^4}{15}\le H^m(z)$.
Finally, as $\tilde I=\tilde I_l+\tilde I_r$, $\int_{\R_+}\tilde I_{l,r}$ are also bounded from above and from below (being positive), so that converging subsequences exist.

\subsection{Monotony}
 First notice that $T\to B_\nu(T)$ is monotone in the sense that $T^m\geq T'^m$ implies ${\tilde B}_\nu(T^m)\ge {\tilde B}_\nu(T'^m)$.  Then observe that, all coefficients being positive,  ${\tilde I}^m_{l,r}\ge {\tilde {I'}}^m_{l,r}$ implies that  ${\tilde I}_{l,r}^{m+1}\ge {\tilde I}'^{m+1}_{l,r}$. More precisely, subtract \eqref{lllr} for ${\tilde {I}}^{m+1}_{l,r}$ from \eqref{lllr} for ${\tilde {I'}}^{m+1}_{l,r}$ and check that it is an equation for the difference $D^{m+1}_{l,r}:={\tilde I}^{m+1}_{l,r}-{\tilde {I'}}^{m+1}_{l,r}$ with positive source,
 \eeqn{\label{DD}&&
 \mu\p_z D_l^{m+1} + \partial_z\log n  (1-\mu^2)\partial_\mu D_l^{m+1}]+\kappa D_l^{m+1}
 \cr&&
\hskip 2cm= \frac{3\beta\kappa_s}8 \int_{-1}^1([2(1-\mu'^2)(1-\mu^2)+\mu'^2\mu^2]D_l^{m} + \mu^2 D_r^m)\d\mu'
\cr&&
\hskip 2cm+ \frac{(1-\beta)\kappa_s}4\int_{-1}^1[D_l^m+ D_r^m]\d\mu'  +  \frac{\kappa_a}2 [\tilde B_\nu({T'}^m)-\tilde B_\nu(T^m)].
 \cr&&
 \mu\p_z  D_r^{m+1} + \partial_z\log n  (1-\mu^2)\partial_\mu D_r^{m+1}+\kappa D_r^{m+1} 
 \cr&&
\hskip 2cm= \frac{3\beta\kappa_s}8 \int_{-1}^1(\mu'^2 D_l^m +  D_r^m)\d\mu' 
 \cr&&
 \hskip 2cm + \frac{(1-\beta)\kappa_s}4\int_{-1}^1[D_l^m+ D_r^m]\d\mu'  +  \frac{\kappa_a}2 [\tilde B_\nu({T'}^m)-\tilde B_\nu(T^m)].
}
 Finally the last equation of \eqref{lllr} implies 
 \eeqn{&
 \int_{\R_+}\kappa_a \tilde B_\nu(T'^{m+1})\d\nu &= \int_{\R_+}\kappa_a \left(\tfrac12\int_{-1}^1(\tilde {I'}_l^{m}+\tilde {I'}_r^{m})\d\mu\right) 
\cr&&
\le 
 \int_{\R_+}\kappa_a \left(\tfrac12\int_{-1}^1(\tilde {I}_l^{m}+\tilde {I}_r^{m})\d\mu\right) 
 =\int_{\R_+}\kappa_a \tilde B_\nu(T'^{m+1})\d\nu,
}
which implies that  $T^{m+1}\ge T'^{m+1}$.
 Let us apply this argument to $\{T^{m-1},{\tilde I}_{i,r}^{m-1}\}$ instead of $\{T'^{m},{\tilde I}'^{m}_{i,r}\}$. It shows that 
 \[
 T^m\ge T^{m-1},~ {\tilde I}^m_{l,r}\ge {\tilde I}_{i,r}^{m-1}~~ \Rightarrow ~~T^{m+1}\ge T^m,~ {\tilde I}^{m+1}_{l,r}\ge {\tilde I}_{i,r}^{m}.
 \]
  To start the iterations, simply set $T^0=0$, ${\tilde I}^0_{l,r}=0$, then by the positivity of the coefficients ${\tilde I}^1_{l,r}\ge 0$ and $T^1\ge 0$.

 The same argument works with  $T^m\le T'^m$, ${\tilde I}^m_{l,r}\le {\tilde I}'^m_{l,r}$ implying that ${\tilde I}^{m+1}_{l,r}\le {\tilde I}'^{m+1}_{l,r}$ and then $T^{m+1}\le T'^{m+1}$. Hence starting with $T^1< T^0$, ${\tilde I}^{1}_{l,r}\le {\tilde I}^{0}_{l,r}$ leads to a decreasing sequence toward the solution. Decreasing sequences being bounded from below by zero, it implies convergence and also the existence of a solution. 
The above results are summarized in the following theorem.

\begin{proposition}
If there is a $I^0_{l,r},T^0$ such that $[I^1_{l,r},T^1]\le [I^0_{l,r},T^0]$, then the solution $T^*,I^*_{l,r}$ of \eqref{lllreq}\eqref{fort} exists and it can be reached numerically from above or below by iterations \eqref{lllr} and these are monotone decreasing and increasing respectively. 
\end{proposition}
We will give a construction of such initialization and it will also imply that the solution is unique.

\section{ Implementation with Integrals}

Consider \eqref{lq}, the system for the irradiance $I$ and the polarization $Q$.
Denote
\eeqn{&&
\tilde J_k(z) = \tfrac12\int_{-1}^1 \mu^k\tilde I\d\mu,
\quad
\tilde K_k(z) = \tfrac12\int_{-1}^1 \mu^k\tilde Q\d\mu .
\quad k=0,2.
}
Then  \eqref{lq} is rewritten as
\begin{align*}\ds 
\mu \p_z \tilde I +\partial_z\log n (1-\mu^2)\partial_\mu\tilde I+ \kappa \tilde I &=\kappa_a \tilde B_\nu + \kappa_s \tilde J_0&
\cr &
+ \frac{\beta\kappa_s}4 P_2(\mu)(3 \tilde J_2 - \tilde J_0 
-3 \tilde K_0 + 3 \tilde K_2),
\\ 
\mu \p_z \tilde Q + \partial_z\log n  (1-\mu^2)\partial_\mu \tilde Q+ \kappa \tilde Q &= 
\cr &-\frac{\beta\kappa_s}4 (1-P_2(\mu))(3\tilde J_2-\tilde J_0 - 3 \tilde K_0 +3 \tilde K_2).
\end{align*}
Let these be multiplied  by $\mu^k$ and integrated in $\mu$. By \eqref{Jdef}
\eeqn{&\ds
\tilde J_k(z) &= \tfrac12\int_{-1}^1\mu^k \tilde I(z,\mu)\d\mu = \frac{c_E}2  \tilde B_\nu(T_E) \E_{k+3}(\kappa, z,0) 
\cr&&+
\frac12\int_0^Z \left(\E_{k+1}(\kappa,z,y)S_0(y)+ \E_{k+3}(\kappa,z,y)S_2(y)\right)\d y.
}
with
\begin{equation}\label{SS}
S_0=\kappa_a \tilde B + \kappa_s \tilde J_0
- \frac{3\beta\kappa_s}{8}( \tilde J_2 - \tfrac13 \tilde J_0 - \tilde K_0 +  \tilde K_2),
\quad
S_2 = \frac{9\beta\kappa_s}{8}( \tilde J_2 -\tfrac13 \tilde J_0 - \tilde K_0 +  \tilde K_2).
\end{equation}
Similarly (recall that $\tilde Q$ is zero at $z=0$),
\[
\tilde K_k(z) =\tfrac12\int_0^Z \left(\E_{k+1}(\kappa,z,y)S'_0(y) +\E_{k+3}(\kappa,z,y)S'_2(y)\right)\d y,
\]
with
\[
S'_0=\frac{9\beta\kappa_s}{8}( \tilde J_2 - \tfrac13 \tilde J_0 
- \tilde K_0 +  \tilde K_2)
\quad
S'_2 = -\frac{9\beta\kappa_s}{8}( \tilde J_2 - \tfrac13 \tilde J_0 - \tilde K_0 +  \tilde K_2).
\]
So at each iteration we only need to compute, for $k=0,2$,
\[
\tilde H_k(\nu,z):=\tfrac9{16}\int_0^Z \E_{k+1}(\kappa,z,{z'})\beta\kappa_s[ \tilde J_2({z'}) - \tfrac13 \tilde J_0({z'}) - \tilde K_0({z'}) +  \tilde K_2({z'})]\d {z'},
\]
and then set
\eeqn{&
\tilde J_k(z) &= \tfrac12 \tilde B_\nu(T_E) \E_{k+3}(\kappa, z,0)
+\tfrac12 \int_0^Z \E_{k+1}(\kappa,z,z')(\kappa_a \tilde B + \kappa_s \tilde J_0)\d z'
 \cr&&
 - \frac13 \tilde H_k +  \tilde H_{k+2},\qquad 
 \tilde K_k(z) =  \tilde H_k - \tilde H_{k+2},
 }
 and update  $T$ by  solving 
 \[
 \int_\R \kappa_a(\tilde B_\nu(T(z))- \tilde J_0(z,\nu))\d\nu=0,~~ \forall z\in(0,Z).
 \]
 \begin{remark}
 Finding $T$ when $J_0$ is given can be done by a Newton method at each $z$. Convergence is implied by the fact that $T\mapsto B_\nu(T)$ is strictly increasing and $\partial_{TT} B_\nu(T)$ is bounded in any interval $[T_m,T_M]$ containing the solution. However, as usual, one must not start too far from the solution and only for frequencies $\nu\ge\nu_m>0$.
 \end{remark}

 \subsection{Formulation with integrals of $I_l,I_r$}
Define
\begin{equation*}
 \tilde  J'_k(z) = \tfrac12\int_{-1}^1 \mu^k  \tilde I_l\d\mu,
\quad
 \tilde  K'_k(z) = \tfrac12\int_{-1}^1 \mu^k  \tilde I_r\d\mu, 
\quad k=0,2, \text{and consider}
\end{equation*}
\vskip-0.5cm 
\begin{equation}
\left\{\begin{aligned}\label{lllreq2}&
 \mu\p_z  { \tilde I}_l +\partial_z\log n (1-\mu^2)\partial_\mu {\tilde I}_l+\kappa { \tilde I}_l 
\cr&\hskip 2cm
= \frac{3\beta\kappa_s}8[2(1-\mu^2)(\tilde J'_0-\tilde J'_2) + \tilde J'_2\mu^2 + \mu^2 \tilde K'_0]
 \cr&
 \hskip 2cm +\frac{(1-\beta)\kappa_s}4[\tilde J'_0+\tilde K'_0]  +  \frac{\kappa_a}2  \tilde B(T(z)), 
 \cr&
 \mu\p_z  { \tilde I}_r +\partial_z\log n (1-\mu^2)\partial_\mu \tilde I_r+\kappa { \tilde I}_r 
 \cr&\hskip 2cm
 =\frac{3\beta\kappa_s}8 [\tilde J'_2+\tilde K'_0] 
  +\frac{(1-\beta)\kappa_s}4[\tilde J'_0+\tilde K'_0] +  \frac{\kappa_a}2  \tilde B(T(z)).
\end{aligned}\right.
\end{equation}
With similar notations as for $[I,Q]^T$, the sources are
\begin{align}\label{sourcep2}&
\tilde\vS=[{\tilde S}_1,{\tilde S}_2]^T= [{{\tilde S}}_1^0+\mu^2{{\tilde S}}_1^2,{{\tilde S}}_2^0+\mu^2{{\tilde S}}_2^2]^T,
\cr&
{{\tilde S}}_1^0 = \frac{3\beta\kappa_s}4[\tilde J'_0-\tilde J'_2]+ \frac{(1-\beta)\kappa_s}4[\tilde J'_0+\tilde K'_0]
+\frac{\kappa_a}2 \tilde B(T(z)),
\cr&
{{\tilde S}}_1^2 = \frac{3\beta\kappa_s}8[3\tilde J'_2-2\tilde J'_0  + \tilde K'_0],
\quad
{{\tilde S}}_2^2=0,
\cr&
{{\tilde S}}_2^0 = \frac{3\beta\kappa_s}8[\tilde J'_2+\tilde K'_0 ] + \frac{(1-\beta)\kappa_s}4[\tilde J'_0+\tilde K'_0]
+ \frac{\kappa_a}2 \tilde  B(T(z)).
\end{align}
%
{\small
\begin{equation*}
\left\{\begin{aligned}
\label{JQ4}&\ds
\tilde J'_q(z) = \tfrac {S_E}4  \tilde B(T_E) \E_{q+3}(\kappa z) 
+
\tfrac12\int_0^Z \left(\E_{q+1}(\kappa|z-y|){\tilde S}^0_1+ \E_{q+3}(\kappa|z-y|){\tilde S}^2_1\right)\d y,
\cr& \ds
\tilde K'_q(z) = \tfrac {S_E}4 \tilde  B(T_E) \E_{q+3}(\kappa z) 
+
\tfrac12\int_0^Z \left(\E_{q+1}(\kappa|z-y|){\tilde S}^0_2+ \E_{q+3}(\kappa|z-y|){\tilde S}^2_2\right)\d y.
\end{aligned}\right.
\end{equation*}}
\begin{remark}
Note that 
\[
\tilde J'_0- \tilde J'_2=\tfrac12\int_{-1}^1(1-{\mu'}^2 )\tilde I_l\d\mu'\ge 0,
\quad 
3\tilde J'_0-2 \tilde J'_2=\int_{-1}^1(3-2{\mu'}^2 )\tilde I_l\d\mu'\ge 0.
\]
Therefore $\tilde\vS$ is always non negative.
\end{remark}
 
 \subsection{A Method to Start with a $T^0$ that leads to a $T^1\le T^0$}\label{start}
 Let us begin with the case $S_E=0$.
 Notice that when $I$ is independent of $\mu$, $\tilde J_2=\tfrac13 \tilde J_0$. So choose  a large constant $T^0$.  Let  $I^0_\nu$ be
 \[
\tilde I^0_\nu := \tilde B_\nu(T^0). \text{ Thence }  \int_{\R_+}\kappa_a \tilde B_\nu(T^0)\d\nu=\int_{\R_+}\kappa_a \tilde J_0^0(\nu)\d\nu,
 \]
and, according to \eqref{SS}, $S_0^0=\kappa_a B_\nu(T^0) + \kappa_s J_0^0=\kappa J_0^0$, $S_2^0=0$, so that
 \begin{align}\label{J0toJ1}
 \tilde J_0^1(z,\nu) = (\kappa_a \tilde B_\nu(T^0) + \kappa_s \tilde J_0^0(\nu))\tfrac12 \int_0^Z E_{1}(\kappa|z-y|)\d y
\cr
= \tilde J_0^0(\nu)\tfrac{\kappa}2 \int_0^Z E_{1}(\kappa|z-y|)\d y
< \tilde J_0^0(\nu).
 \end{align}
 Consequently $T^1(z)\le T^0$ for all $z$, because
 \[
 \int_{\R_+} \kappa_a \tilde B_\nu(T^1) \d\nu =\int_{\R_+} \kappa_a \tilde J_0^1(\nu)\d\nu \le \int_{\R_+} \kappa_a \tilde J_0^0(\nu)\d\nu
 = \int_{\R_+} \kappa_a \tilde B_\nu(T^0) \d\nu. 
 \]
 Now, as  $\S=\kappa \tilde J_0^0$ in  \eqref{SS} it implies that $\tilde I^1<\tilde I^0$.
 Furthermore in the equation for $\tilde Q$ in \eqref{lq}, $I^0$ independent of $\mu$, $\tilde \tilde Q^0=0$ leads to a null source. So $\tilde Q^1=0$.  This means that choosing $\tilde I^0_l=\tilde I^0_r=\tfrac12 \tilde I^0$ (i.e. $\tilde Q^0=0$), 
 leads to $\tilde I^1_l=\tilde I^1_r=\tfrac12 \tilde I^1 \le \tfrac12 \tilde I^0 =\tilde I^0_l=\tilde I^0_r$.
 When $S_E\neq 0$ we can exploit the fact that the inequality in \eqref{J0toJ1} is strict. Now it is
 \[  \tilde J_0^1(z,\nu) \le S_E B_\nu(T_E)E_3(\kappa z) + (1-E_2(\kappa Z))\tilde J_0^0(\nu) < \tilde J_0^0(\nu) \text{ if $\tilde I^0$ is large enough}.
 \]

\begin{theorem}
The solution of the VVRTE with varying index exists and is unique.
\end{theorem}
Existence is a consequence of a decreasing sequence bounded from below by zero. Uniqueness is shown by taking 2 solutions and applying the argument used in \eqref{DD}.
 \section{Extension to non Homogeneous Boundary Conditions at $Z$}
It is straightforward to extend the previous analysis to the following boundary conditions on I
\begin{equation*}
\tilde\vI(0,\mu)=\tilde\vI_0(\mu)\ge 0 \hbox{ for all }\mu>0,
\quad 
\tilde\vI(Z,\mu)= \tilde\vI_Z(\mu)\ge 0 \hbox{ for all }\mu<0.%
\end{equation*}
Then \eqref{gensol} becomes
\begin{align*}
\tilde\vI(z(s),\mu(s)) &= \One_{\mu_0>0}\left[\e^{-\int_0^s\frac{\kappa(s')}{\mu(s')}d s'}\tilde\vI_0(\mu_0) 
+ \int_0^s\frac{\e^{-\int^s_{s'} \frac{\kappa(s'')}{\mu(s'')}d s''}}{\mu(s')}{\tilde\vS(s')}\d s'\right]
\cr&
+\One_{\mu_0<0}\left[\e^{\int_s^Z\frac{\kappa(s')}{\mu(s')}d s'}\tilde\vI_Z(\mu_Z) 
-\int_s^S \frac{\e^{\int_{s}^{s'}\frac{\kappa(s'')}{\mu(s'')} d s'' }}{\mu(s')}{\tilde\vS(s',\mu(s'))}\d s'\right].
\end{align*}
For atmospheres, $\tilde\vI_Z$ comes from the sun, $\tilde\vI_Z(\mu)=-\mu c_S\tilde B_\nu(T_S)$, with $c_S$ the light intensity of the sun and $T_S$ the temperature of the sun, felt at $z=Z$.
The only change to the algorithm is that now
\eeqn{&\ds
\tilde J_k(z) &= \frac{c_E}2  \tilde B_\nu(T_E) \E_{k+3}(\kappa, z,0) 
+ \frac{c_S}2  \tilde B_\nu(T_S) \E_{k+3}(\kappa, z,Z) 
\cr&&+
\frac12\int_0^Z \left[\E_{k+1}(\kappa,z,y)S_0(y)+ \E_{k+3}(\kappa,z,y)S_2(y)\right]\d y.
}
 \section{Precision}
 \subsection{An approximation to use Exponential Integrals}
Notice that $n^2_{z'}=n^2_{z}(1+\epsilon)$,~$\epsilon << 1$ implies $\eta^2(\mu,z,z')=(1+\epsilon)\mu^2(z)-\epsilon$.
The computer implementation is easier and fast if  the $-\epsilon$ is dropped, i.e. if
\eq{\label{approxmu}
\eta(\mu,z,z')\approx \eta_h:=\mu\frac{n_{z'}}{n_z},
}
because then exponential integrals can be used to approximate $\E_k$ defined in \eqref{falseE}.  Figure \ref{approxmun} shows the error between $\eta$ and $\eta_h$.
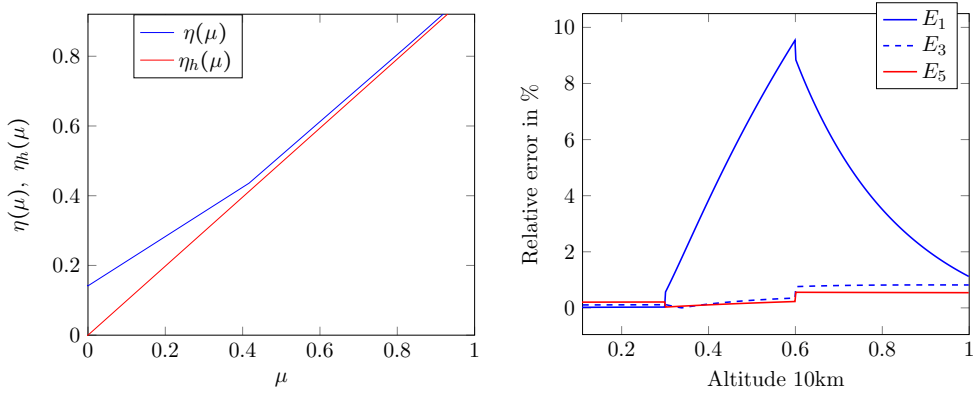
\begin{figure}[htbp]
\begin{minipage} [b]{0.45\textwidth}
\begin{center}
 \begin{tikzpicture}[scale=0.75]
    \begin{axis}[legend style={at={(0.4,0.9)},anchor= east},
   xmin=0, xmax=1, ymin=0,
   xlabel= {$\mu$},
  ylabel= {$\eta(\mu),~~\eta_h(\mu)$}
]
        \addplot[color=blue] {sqrt(x*x*0.98+0.02) };
        \addlegendentry{$\eta(\mu)$}
		 \addplot[color=red] {x*0.99};
		 \addlegendentry{$\eta_h(\mu)$}
    \end{axis}
\end{tikzpicture}       
\end{center}
\end{minipage}
\hskip1.cm
\begin{minipage} [b]{0.45\textwidth}
\begin{center}
\begin{tikzpicture}[scale=0.75]
\begin{axis}[legend style={at={(0.97,0.9)},anchor= east}, compat=1.3,
   xmin=0.11, xmax=1,
   xlabel= {Altitude 10km},
  ylabel= {Relative error in \%}
  ]
\addplot[thick,solid,color=blue,mark=none, mark size=1pt] table [x index=0, y index=1]{fig/compareN.txt};
\addlegendentry{$E_1$}
%
\addplot[thick,dashed,color=blue,mark=none, mark size=1pt] table [x index=0, y index=2]{fig/compareN.txt};
\addlegendentry{$E_3$ }
\addplot[thick,solid,color=red,mark=none, mark size=1pt] table [x index=0, y index=3]{fig/compareN.txt};
\addlegendentry{$E_5$}
\end{axis}
\end{tikzpicture}
\end{center}
\end{minipage}
\caption{\footnotesize LEFT:\label{approxmun} Difference between $\eta_h(\mu)$ and $\eta(\mu)$ (see \eqref{approxmu}) when $n_{z'}/n_z=0.99$.
RIGHT:\label{approxen} \label{errorEk}  Relative error due to \eqref{approxE} for  $\E_k$, $k=1,3,5$. }

\end{figure}

Indeed, if $E_k$ is the $k^{th}$ exponential integral,
\begin{equation}\label{approxE}
\E_k(\mu,z,z') \approx\left(\frac{n_{z'}}{n_z}\right)^{k-2}E_k(\int^z_{z'} \kappa(z'')\frac{n_z}{n_{z''}}\d z'').
\end{equation}
However figure \ref{approxen} shows that this approximation is much too coarse for $\E_1$ but feasible for $\E_3$ and $\E_5$.

To optimize the computing time, $\E$ is tabulated in an array for 100 values of $z,z'$ and 50 values of $\kappa_\nu$. Computing this array takes 5'' and the rest of the program runs faster than with exponential integrals.

\subsection{Quadrature}
The iterations converge rather fast and the solution can be bounded from above and below by the decreasing and increasing sequences.  The Newton iterations to compute the temperature from the knowledge of $I$ can also be driven to machine precision with a small number of iterations because $T\mapsto B_\nu(T)$ is strictly increasing.
The bottleneck is the precision to compute integrals such as
\[\ds
J_k(z,z') = \int_{[0,1]\cap M}\mu^{k-1}\frac{\exp\left( -\kappa_\nu\int_z^{z'}\kappa(y)\left(1-(1-\mu^2)\frac{n^2(z)}{n^2(y)}\right)_+^{-\frac12}dy\right)}{\left(1-(1-\mu^2)\frac{n^2(z)}{n^2(z')}\right)^{\frac12}}\d\mu
\]
with $M=\{\mu~:~1-(1-\mu^2)\frac{n^2(z)}{n^2(z')}>0\}$.  There is a singularity at $\mu=\sqrt{1-\frac{n^2(z')}{n^2(z)}}$ if non-negative.

We use a quadrature formula at $\mu^j=(j\delta\mu)^2$ if $\mu^j<\mu^*$ and $\mu^j=j\delta\mu$ if $\mu^j>\mu^*$.

When $n$ is constant, $J_k(0,z')$ is the exponential integral $E_k(\kappa_\nu\int_0^{z'}\kappa(y)d y)$ for which there is a very precise approximated formula when $\kappa_nu\kappa(y)$ is not large \cite{ABR}.
With $\kappa_\nu=0.5$ and $\kappa(z)=1-z/2$, figure \ref{convergeExp} displays the precision obtained on $E_1,E_3,E_5$ when $\delta\mu=0.02,0.01,0.05$ and $\mu^*=0.1$. Already with $\mu=0.01$ the relative precision is less than $1\%$.The integral in the exponential is computed with a fixed increment $\delta z=1/60$.

When $n$ is not constant we can only observe the convergence towards the value obtained with a very small $\delta\mu$ and $\delta z$, as shown in figure \ref{convergeExp2}.
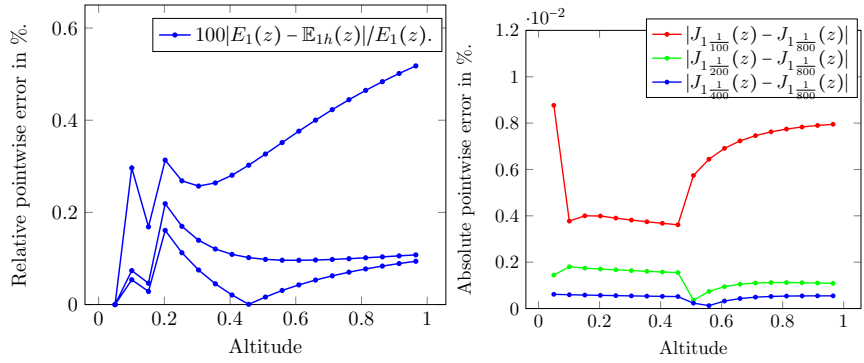
\begin{figure}[htbp]
\begin{minipage} [b]{0.45\textwidth}. 
\begin{center}
\begin{tikzpicture}[scale=0.7]
\begin{axis}[legend style={at={(0.99,0.9)},anchor= east}, 
   ymin=0., ymax=0.65,
   xlabel= {Altitude},
   ylabel = {Relative pointwise error in \%.}
  ]
\addplot[thick,solid,color=blue,mark=*, mark size=1pt] table [x index=0, y index=1]{fig/test.txt};
\addlegendentry{$100|E_1(z)-\E_{1h}(z)|/E_1(z)$.}
\end{axis}
\end{tikzpicture}
\end{center}
\end{minipage}
%
%
\begin{minipage} [b]{0.45\textwidth}. 
\begin{center}
\begin{tikzpicture}[scale=0.65]
\begin{axis}[legend style={at={(0.99,0.9)},anchor= east}, 
   ymin=0., ymax=0.012,
   xlabel= {Altitude},
   ylabel = {Absolute pointwise error in \%.}
  ]
\addplot[thick,solid,color=red,mark=*, mark size=1pt] table [x index=0, y index=1]{fig/ntest.txt};
\addlegendentry{$|J_{1\frac1{100}}(z)-J_{1\frac1{800}}(z)|$}
\addplot[thick,solid,color=green,mark=*, mark size=1pt] table [x index=0, y index=2]{fig/ntest.txt};
\addlegendentry{$|J_{1\frac1{200}}(z)-J_{1\frac1{800}}(z)|$}
\addplot[thick,solid,color=blue,mark=*, mark size=1pt] table [x index=0, y index=3]{fig/ntest.txt};
\addlegendentry{$|J_{1\frac1{400}}(z)-J_{1\frac1{800}}(z)|$}
%
\end{axis}
\end{tikzpicture}
\end{center}
\end{minipage}

\caption{\footnotesize LEFT:  \label{convergeExp} Convergence of the approximate exponential integrals $E_{1h}$ to $E_{1}$ versus $\d\mu=0.02,0.01,0.05$ (each curve is below the previous one). The corresponding top,middle and bottom curves of the pointwise relative errors are plotted versus $z$. At $\d\mu=0.01$ the 3 relative errors are below 1\%. $\d\mu=0.005$ does not improve the precision. 
RIGHT:
 \label{convergeExp2} Absolute  error between $J_{1\delta\mu}$ and $J_{1\delta\mu_{min}}$ versus $z$ when $\delta\mu=0.01,0.005,0.0025$ and $\delta\mu_{min}=1/800$. Here $n_m$ and $\kappa=0.5$ are used. 
}\end{figure}

\subsection{Numerical Results}
According to \cite{measureN} the variation of the refractive index in the atmosphere is quite small $\sim 0.003$. To enhance the effect we use 3 times this value.

We investigated 2  case:
\begin{itemize}
\item
Case 1: IR light coming from Earth with or without \texttt{CO}$_2$ effects with constant $\kappa=0.5$ or the values from Gemini experiment \footnote{
\url{www.gemini.edu/observing/telescopes-and-sites/sites\#Transmission}
}
.
\item
Case 2: Visible light coming from the sun hits the top of the troposphere, with or without \texttt{CO}$_2$ effects with constant $\kappa=0.5$ or the values from Gemini experiment..
\end{itemize}
For all tests the following is used:
\begin{itemize}
\item $n(z)= 1+\epsilon \One_{z\in(0.5,0.7)}$,
~$\epsilon=0.01$ or $0$,
\label{nnn}
\item $\beta=0.5$, $a_s=a_1\One_{z\in(z_1,z_2)} + a_2\One_{z>z_2}\One_{\nu\in(\nu_1,\nu_2)}\left(\frac\nu{\nu_2}\right)^4$, 
\item  $a_1=0.7,~a_2=0.3$, $z_1=0.4,~z_2=0.8$, $\nu_1=0.6,~\nu_2=1.5$.
\item $c_S = 2.10^{-5}, T_S = \frac{5700}{4798}, c_E = 2.5, T_E =\frac{300}{4798}$. 
\end{itemize}
The monotony of the iterative process is displayed in Figure \ref{convergefig}.  It is clear that by starting below (resp. above) the exact solution, the values of the temperature at $z=300$m are increasing (resp. decreasing). Note that 15 iterations are sufficient to obtain a 3 digits precision.

To study the effect of $n$ on a simple case with ran the program with $\kappa=0.5$, $n$ as in \ref{nnn} and $\epsilon=0$ and the data of Case 1. The results are shown in Figures \ref{tempeK2}.
In all other test cases $\kappa$ is the function of $\nu$ extracted from the Gemini experiments (Figure \ref{kappafig}).
In Figure \ref{tempe} Temperature versus altitude is displayed for Case 1 \& 2 for 2 different $\nu\to\kappa$, the Gemini values and the Gemini $\nu\to\kappa_1$ modified due to an increase of \texttt{CO}$_2$ as shown in Figure \ref{kappafig}.
The main points are
\begin{itemize}
\item
For Case 1 (IR light coming from Earth) the \texttt{CO}$_2$ increases the temperature by $0.5^o$C near the ground and  decreases it by a similar amount above 5000m.
\item
For Case 2 (Visible light coming from the sun hits the top of the troposphere) the effect of the \texttt{CO}$_2$ is a drastic reduction of temperature, both near the ground and in altitude.
\item
In both cases the influence of the cloud is seen as a decrease of, the temperature in the cloud and above it for Case 1.
\end{itemize}
In Figure \ref{lightIK1}, with $\kappa$-Gemini, the integrals of intensities over all ray directions are shown, namely $\nu\to J_0$ and the polarization $\nu\to K_0$ at ground level.  $J_0$ increases with altitude while $K_0$ decreases. As $10^{7}K_0(Z)$ has  large negative values at some points we have displayed $\tilde K_0=\max(K_0,-2\, 10^{-6})$.

In Figure \ref{lightIK1}  the effect of adding an added opacity in the range $14-18\mu$m is seen on $J_0$ and $K_0$.

Notice that the polarization is particularly strong at ground level near $\nu=\tfrac3{18}$ and since $Q(0,\mu)=0$ when $\mu>0$ it is entirely due to rays pointing downward.

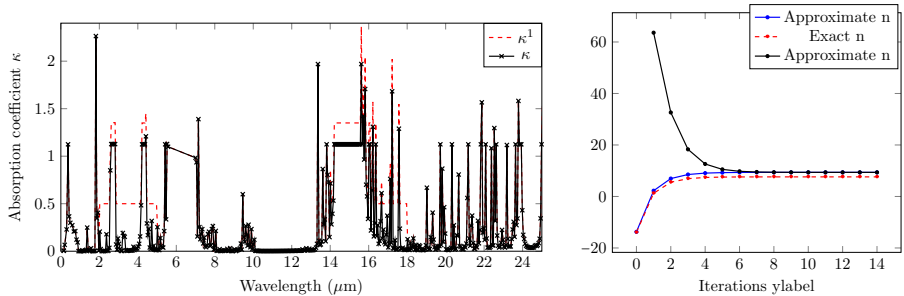
\begin{figure}[htbp]
\begin{minipage} [b]{0.45\textwidth}
\begin{center}
\begin{tikzpicture}[scale=0.56]
\begin{axis}[width=13cm,height=7cm,legend style={at={(1,1)},anchor=north east}, compat=1.3,
   xmin=0,xmax=25, ymin=0, ymax=2.4,
   xlabel= {Wavelength ($\mu$m)},
  ylabel= {Absorption coefficient  $\kappa$}
  ]
%
\addplot[thick,dashed,color=red,mark=none, mark size=1pt] table [x index=0, y index=1]{fig/kappa1.txt};
\addlegendentry{ $\kappa^1$}
\addplot[thick,solid,color=black,mark=none, mark=x] table [x index=0, y index=1]{fig/kappa.txt};
\addlegendentry{ $\kappa$}
\end{axis}
\end{tikzpicture}
\end{center}
\end{minipage}
\hskip1.1cm
\begin{minipage} [b]{0.45\textwidth}. 
\begin{center}
\begin{tikzpicture}[scale=0.56]
\begin{axis}[legend style={at={(0.99,0.9)},anchor= east}, 
   xlabel= {Iterations}
   ylabel = {$T_-^m|_{z=500m}$ in $C^o$}
  ]
\addplot[thick,solid,color=blue,mark=*, mark size=1pt] table [x index=0, y index=1]{fig/approxNgrow.txt};
\addlegendentry{Approximate n}
\addplot[thick,dashed,color=red,mark=*, mark size=1pt] table [x index=0, y index=1]{fig/exactNgrow.txt };
\addlegendentry{Exact n}
\addplot[thick,solid,color=black,mark=*, mark size=1pt] table [x index=0, y index=1]{fig/approxNdecrease.txt};
\addlegendentry{Approximate n}
%
\end{axis}
\end{tikzpicture}
\end{center}
\end{minipage}
\caption{ \footnotesize LEFT:\label{kappafig} Absorption $\kappa$ from the Gemini experiment, versus wavenumber ($3/\nu$). In dotted lines, the modification to construct $\kappa_1$ to account for the opacity of \texttt{CO}$_2$.
RIGHT:\label{convergefig} Convergence of the temperature at altitude 300m during the iterations. In solid line when it is started with $T^0=0$, in dashed line when it intial temperature is 180°C. Notice the monotonicity of both curves.
}\end{figure}

\begin{figure}[htbp]
\begin{minipage} [b]{0.45\textwidth}
\begin{center}
\begin{tikzpicture}[scale=0.7]
\begin{axis}[legend style={at={(0.97,0.9)},anchor= east}, compat=1.3,
   xmin=0.11, xmax=1,
   xlabel= {Altitude 10km},
  ylabel= {Temperature $^o$C}
  ]
\addplot[thick,solid,color=blue,mark=none, mark size=1pt] table [x index=0, y index=1]{fig/temperature21.txt};
\addlegendentry{$\epsilon=0.01$}
\addplot[thick,dashed,color=blue,mark=none, mark size=1pt] table [x index=0, y index=1]{fig/temperature22.txt};
\addlegendentry{$\epsilon=0.$ }
\end{axis}
\end{tikzpicture}
\end{center}
\end{minipage}
\hskip0.5cm
\begin{minipage} [b]{0.45\textwidth}. 
\begin{center}
\begin{tikzpicture}[scale=0.7]
\begin{axis}[legend style={at={(0.5,0.8)},anchor= east}, compat=1.3,
  xmin=-5,ymax=20,
   ylabel= {light-intensity},
  xlabel= {wave length $\mu$m}
  ]
\addplot[thick,dashed,color=red,mark=none, mark size=1pt] table [x index=0, y index=1]{fig/imean22Z.txt};
\addlegendentry{$10^5J_0(Z), \epsilon=0.01$}
\addplot[thick,solid,color=red,mark=none, mark size=1pt] table [x index=0, y index=1]{fig/imean21Z.txt};
\addlegendentry{$10^7\tilde J_0(Z), \epsilon=0.$}
\addplot[thick,dashed,color=blue,mark=none, mark size=1pt] table [x index=0, y index=2]{fig/imean22Z.txt};
\addlegendentry{$10^5K_0(Z), \epsilon=0.01$}
\addplot[thick,solid,color=blue,mark=none, mark size=1pt] table [x index=0, y index=2]{fig/imean21Z.txt};
\addlegendentry{$10^7\tilde K_0(Z), \epsilon=0.$}
\end{axis}
\end{tikzpicture}
\end{center}
\end{minipage}
\caption{ \footnotesize  \label{lightI2} {\bf Case 1} with $\kappa=0.5$, $n(z)=1+\epsilon\One_{z\in(0.5,0.7)}$.  
LEFT: \label{tempeK2}  Temperatures versus altitude:  the dashed curve is computed with $\epsilon=0$, the solid curve with $\epsilon=0.01$. It is colder in the cloud.
RIGHT: Total light and polarized intensity $J_0$  and  $K_0$ versus wave length at altitude Z=10km with $\epsilon=0.$ and $\epsilon=0.01$.
}
\end{figure}
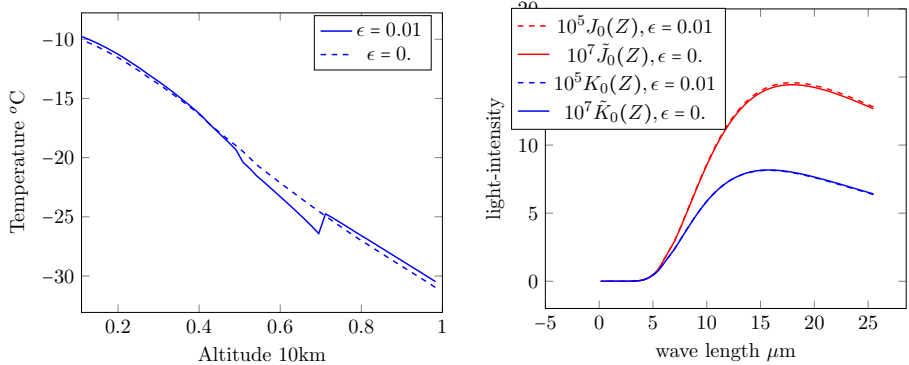

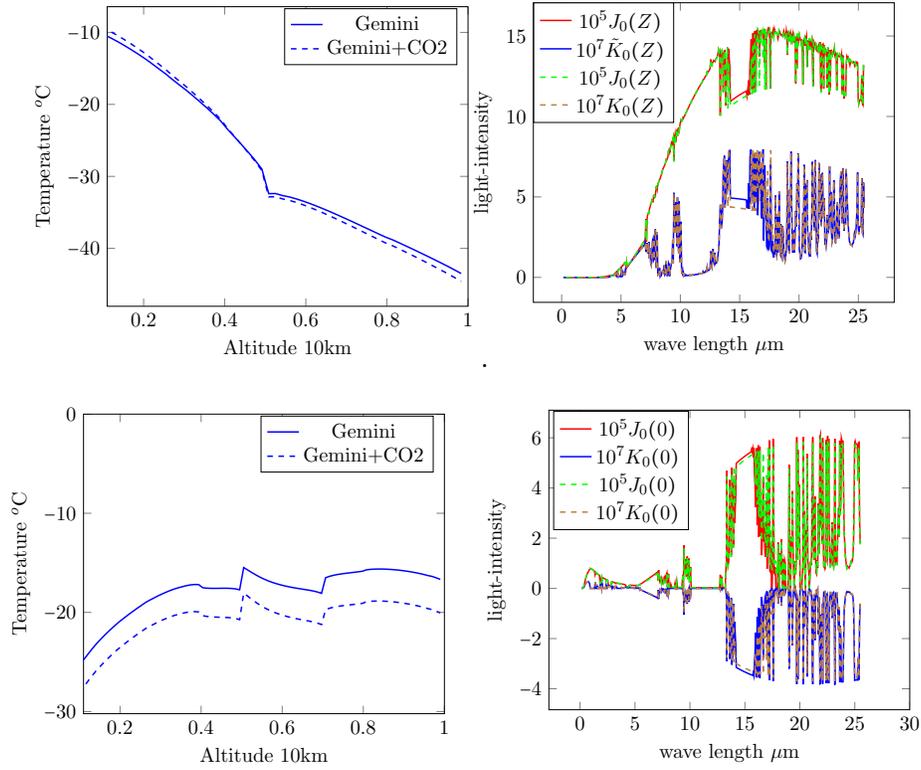
\begin{figure}[htbp]
\begin{minipage} [b]{0.45\textwidth}
\begin{center}
\begin{tikzpicture}[scale=0.7]
\begin{axis}[legend style={at={(0.97,0.9)},anchor= east}, compat=1.3,
   xmin=0.11, xmax=1,
   xlabel= {Altitude 10km},
  ylabel= {Temperature $^o$C}
  ]
\addplot[thick,solid,color=blue,mark=none, mark size=1pt] table [x index=0, y index=1]{fig/temperature1.txt};
\addlegendentry{Gemini}
%
\addplot[thick,dashed,color=blue,mark=none, mark size=1pt] table [x index=0, y index=1]{fig/temperature11.txt};
\addlegendentry{Gemini+CO2 }
\end{axis}
\end{tikzpicture}
\end{center}
\end{minipage}
\begin{minipage} [b]{0.45\textwidth}. 
\begin{center}
\begin{tikzpicture}[scale=0.7]
\begin{axis}[legend style={at={(0,0.8)},anchor= west}, compat=1.3,
   ylabel= {light-intensity},
  xlabel= {wave length $\mu$m}
  ]
\addplot[thick,solid,color=red,mark=none, mark size=1pt] table [x index=0, y index=1]{fig/imean2Z.txt};
\addlegendentry{$10^5J_0(Z)$}
\addplot[thick,solid,color=blue,mark=none, mark size=1pt] table [x index=0, y index=2]{fig/imean2Z.txt};
\addlegendentry{$10^7\tilde K_0(Z)$}
\addplot[thick,dashed,color=green,mark=none, mark size=1pt] table [x index=0, y index=1]{fig/imean12Z.txt};
\addlegendentry{$10^5J_0(Z)$}
\addplot[thick,dashed,color=brown,mark=none, mark size=1pt] table [x index=0, y index=2]{fig/imean12Z.txt};
\addlegendentry{$10^7K_0(Z)$}
\end{axis}
\end{tikzpicture}
\end{center}
\end{minipage}
%
\begin{minipage} [b]{0.45\textwidth}
\begin{center}
\begin{tikzpicture}[scale=0.7]
\begin{axis}[legend style={at={(0.97,0.9)},anchor= east}, compat=1.3,
   xmin=0.11, xmax=1, ymax=0,
   xlabel= {Altitude 10km},
  ylabel= {Temperature $^o$C}
  ]
\addplot[thick,solid,color=blue,mark=none, mark size=1pt] table [x index=0, y index=1]{fig/temperature102.txt};
\addlegendentry{Gemini}
\addplot[thick,dashed,color=blue,mark=none, mark size=1pt] table [x index=0, y index=1]{fig/temperature112.txt};
\addlegendentry{Gemini+CO2 }
\end{axis}
\end{tikzpicture}
\end{center}
\end{minipage}
\hskip0.5cm
\begin{minipage} [b]{0.45\textwidth}. 
\begin{center}
\begin{tikzpicture}[scale=0.7]
\begin{axis}[legend style={at={(0.2,0.6)},anchor= south}, compat=1.3,
  xmax=30,
   ylabel= {light-intensity},
  xlabel= {wave length $\mu$m}
  ]
\addplot[thick,solid,color=red,mark=none, mark size=1pt] table [x index=0, y index=1]{fig/imean1020.txt};
\addlegendentry{$10^5J_0(0)$}
\addplot[thick,solid,color=blue,mark=none, mark size=1pt] table [x index=0, y index=2]{fig/imean1020.txt};
\addlegendentry{$10^7 K_0(0)$}
\addplot[thick,dashed,color=green,mark=none, mark size=1pt] table [x index=0, y index=1]{fig/imean1120.txt};
\addlegendentry{$10^5J_0(0)$}
\addplot[thick,dashed,color=brown,mark=none, mark size=1pt] table [x index=0, y index=2]{fig/imean1120.txt};
\addlegendentry{$10^7K_0(0)$}
\end{axis}
\end{tikzpicture}
\end{center}
\end{minipage}
\caption{$n(z)=1+\epsilon\One_{z\in(0.5,0.7)}$. LEFT: \footnotesize \label{tempe} {\bf Case 1} (top) \& {\bf Case 2}  (bottom). Temperature versus altitude; the dashed curves are computed with $\kappa_1$ to account for  $\texttt{CO}_2$. RIGHT:  \label{lightIK1} {\bf Case 1} (top) \& {\bf Case 2} (bottom) with $\kappa$ (solid) and $\kappa_1$ (dashed). Total light and polarized intensities $J_0$,  $K_0$ versus wave length.}
\end{figure}

\subsection{Computation of $I$ and $Q$ versus $z$ and $\mu$}
For each $\nu$ at which  $I$ and $Q$ are desired we use \eqref{gensol} with \eqref{SS}. The computation is fast but some of the integrals have singularities, so it needs to be implemented with care.

We have computed {Case 1} with $\kappa=0.5$ and uniform scattering. Results are shown for $\nu=0.1436$ on figures  \ref{Kzmu1} for $\epsilon=0$ and on  figures  \ref{Kzmu2} for $\epsilon=0.01$.

\begin{figure}[htbp]
\begin{minipage} [b]{0.47\textwidth}
\centering
\begin{tikzpicture}[scale=0.8]
\begin{axis}[ legend style={at={(1,1)},anchor=north east}, compat=1.3,xlabel= {$z$},ylabel= {$\mu$}]
 \addplot3[surf,fill opacity=0.75] table [ ] {fig/gnuplotI0.txt};
\addlegendentry{ $10^5\cdot I(z,\mu)$}
\end{axis}
\end{tikzpicture}
\end{minipage}
\hskip 0.25cm
\begin{minipage} [b]{0.47\textwidth}
\centering
\begin{tikzpicture}[scale=0.8]
\begin{axis}[legend style={at={(1,1)},anchor=north east}, compat=1.3,xlabel= {$z$},ylabel= {$\mu$}]
 \addplot3[surf,fill opacity=0.5] table [ ] {fig/gnuplotK0.txt};
\addlegendentry{  $10^5\cdot Q(z,\mu)$}
\end{axis}
\end{tikzpicture}
\end{minipage}%
\caption{\label{Kzmu1}\footnotesize Case 1 with $n=1$. LEFT: light intensity $I$ versus $z,\mu$. RIGHT: Case 2: Polarization intensity $Q$ versus $z,\mu$, mostly due to $\beta=0.5$.}
\end{figure}
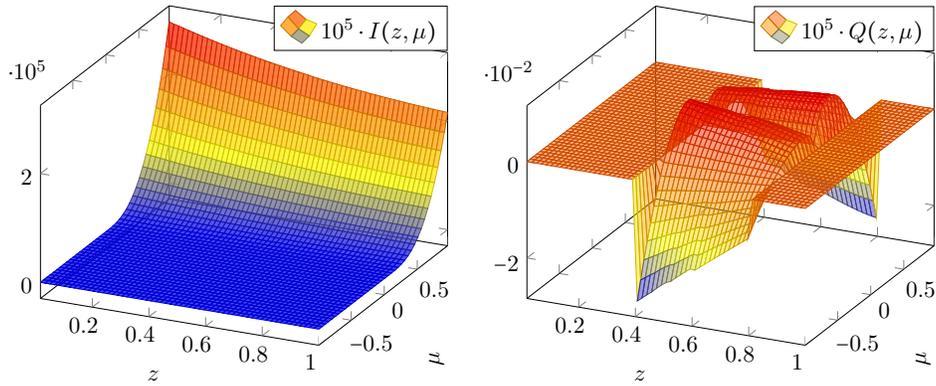

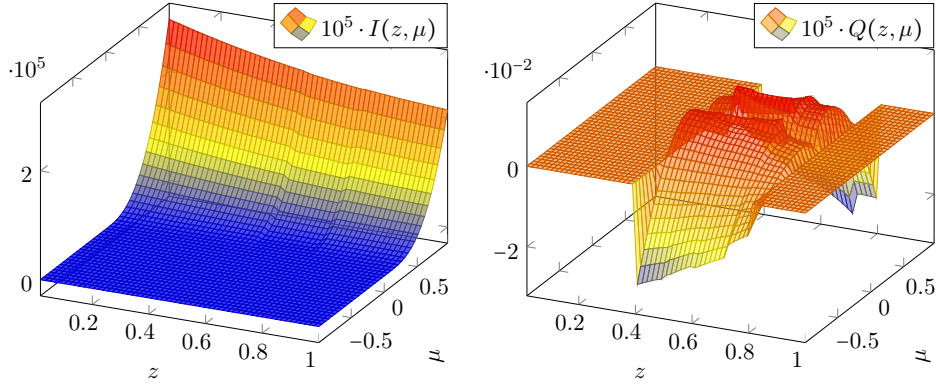
\begin{figure}[htbp]
\begin{minipage} [b]{0.47\textwidth}
\centering
\begin{tikzpicture}[scale=0.8]
\begin{axis}[ legend style={at={(1,1)},anchor=north east}, compat=1.3,xlabel= {$z$},ylabel= {$\mu$}]
 \addplot3[surf,fill opacity=0.75] table [ ] {fig/gnuplotI.txt};
\addlegendentry{ $10^5\cdot I(z,\mu)$}
\end{axis}
\end{tikzpicture}
\end{minipage}
\hskip 0.25cm
\begin{minipage} [b]{0.47\textwidth}
\centering
\begin{tikzpicture}[scale=0.8]
\begin{axis}[legend style={at={(1,1)},anchor=north east}, compat=1.3,xlabel= {$z$},ylabel= {$\mu$}]
 \addplot3[surf,fill opacity=0.5] table [ ] {fig/gnuplotK.txt};
\addlegendentry{  $10^5\cdot Q(z,\mu)$}
\end{axis}
\end{tikzpicture}
\end{minipage}%
\caption{\footnotesize Case 1 with $n=1+0.01*\One_{z\in(0.5,0.7)Z}$. LEFT: light intensity $I$ versus $z,\mu$. RIGHT: Case 2: Polarization intensity $Q$ versus $z,\mu$.
\label{Kzmu2}}
\end{figure}

\section{Conclusion}
In this article the methodology developed in (Golse-Pironneau (2022)\cite{FGOP3}) for the numerical solution of the RTE has been extended to include Rayleigh scattering with polarization and continuous refraction.
The equations have been shown to be well posed when total refraction is ruled out and the numerical method based on ``Iterations on the Source" has been shown to be monotone and convergent when a condition on the data is satisfied. The same condition implies existence and uniqueness of the solution of the VVRTE system.

The method is not hard to program and the execution time is a few seconds.  The opacity of \texttt{CO}$_2$ has a greenhouse effect near the ground and a cooling effect at high altitude. Refraction in clouds cools  air in the clouds and above it when IR comes from the ground. Also \texttt{CO}$_2$ cools the atmosphere everywhere when the radiation comes from the sun and so does clouds. Effects on atmospheric temperatures are small but within the precision of the numerical method and it shows that the variations of the refractive index cannot be neglected. Whether this modeling of the atmosphere is sufficient to explain the greenhouse effect of \texttt{CO}$_2$ is debatable and left to the climatologist (see for example Dufresne et al.(2020)\cite{DUF}).

Generalization to 3D as in (Hecht et al. (2022)\cite{JCP}) and (Pironneau-Tournier (2023)\cite{JCP2}) for a non-stratified atmosphere seems doable.

\section{Appendix: General Rayleigh scattering Matrix\label{appendix}}

According to equation ({219}) in \cite{CHA}p42  we may express the phase-matrix in the form
$$
\Z\left(\mu, \varphi ; \mu^{\prime}, \varphi^{\prime}\right)
=\boldsymbol{Q}[P^0+(1-\mu^2)\frac12(1-\mu')^\frac12 \boldsymbol{P}^{(1)}+\boldsymbol{P}^{(2)}\left(\mu, \varphi ; \mu^{\prime}, \varphi^{\prime}\right)]
$$
$$
\begin{array}{l}
\boldsymbol Q=\left(\begin{array}{cccc}
1 & 0 & 0 & 0 \\
0 & 1 & 0 & 0 \\
0 & 0 & 2 & 0 \\
0 & 0 & 0 & 2
\end{array}\right) 
\quad
P^0=\frac34\left(
\begin{array}{cccc}
2\left(1-\mu^2\right)\left(1-\mu^{\prime 2}\right)+\mu^2 \mu^{\prime 2} & \mu^2 & 0 & 0 \\
\mu^{\prime 2} & 1 & 0 & 0 \\
0 & 0 & 0 & 0 \\
0 & 0 & 0 & \mu \mu^{\prime}
\end{array}
\right),\\
$$

$$\begin{array}{l}
\boldsymbol{P}^{(1)}=
\frac{3}{4}\left(\begin{array}{cccc}
4 \mu \mu^{\prime} \cos \left(\varphi^{\prime}-\varphi\right) & 0 & 2 \mu \sin \left(\varphi^{\prime}-\varphi\right) & 0 \\
0 & 0 & 0 & 0 \\
-2 \mu^{\prime} \sin \left(\varphi^{\prime}-\varphi\right) & 0 & \cos \left(\varphi^{\prime}-\varphi\right) & 0 \\
0 & 0 & 0 & \cos \left(\varphi^{\prime}-\varphi\right)
\end{array}\right)
\end{array}\\
$$

$$
\boldsymbol{P}^{(2)}
=\frac{3}{4}\left(\begin{array}{cccc}
\mu^2 \mu^{\prime 2} \cos 2\left(\varphi^{\prime}-\varphi\right) & -\mu^2 \cos 2\left(\varphi^{\prime}-\varphi\right) & \mu^2 \mu^{\prime} \sin 2\left(\varphi^{\prime}-\varphi\right) & 0 \\
-\mu^{\prime 2} \cos 2\left(\varphi^{\prime}-\varphi\right) & \cos 2\left(\varphi^{\prime}-\varphi\right) & -\mu^{\prime} \sin 2\left(\varphi^{\prime}-\varphi\right) & 0 \\
-\mu \mu^{\prime 2} \sin 2\left(\varphi^{\prime}-\varphi\right) & \mu \sin 2\left(\varphi^{\prime}-\varphi\right) & \mu \mu^{\prime} \cos 2\left(\varphi^{\prime}-\varphi\right) & 0 \\
0 & 0 & 0 & 0
\end{array}\right) \text {, } \\
\end{array}
$$

\bibliographystyle{plain}

\bibliography{references}

\begin{thebibliography}{10}

\bibitem{ABR}
M.~Abramowitz and I.~Stegun.
\newblock {\em Handbook of Mathematical Functions}.
\newblock Dover Publications, Washington D.C., 1972.

\bibitem{graphics}
M.~Ament, C.~Bergmann, and D.~Weiskopf.
\newblock Refractive radiative transfer equation.
\newblock {\em ACM Transactions on Graphics}, 33(2):2, March 2014.

\bibitem{measureN}
H.~Bussey and G.~Birnbaum.
\newblock Measurement of variations in atmosphereic refractive index xith an
  airpborne microwave refractometer.
\newblock {\em Jornal of National Research of the National Bureau of
  Standards}, 51(4):171--178, 1953.

\bibitem{CHA}
S.~Chandrasekhar.
\newblock {\em {Radiative Transfer}}.
\newblock Clarendon Press, Oxford, 1950.

\bibitem{DUF}
J.~Dufresne, V.~Eymet, C.~Crevoisier, and J.~Grandpeix.
\newblock Greenhouse effect: The relative contributions of emission height and
  total absorption.
\newblock {\em Journal of Climate, American Meteorological Society},
  33(9):3827--3844, 2020.

\bibitem{garcia}
R.~Garcia.
\newblock Boundary and interface conditions for polarized radiation transport
  in a multilayer medium.
\newblock In American~Nuclear Society, editor, {\em International Conference on
  Mathematics and Computational Methods Applied to Nuclear Science and
  Engineering}, volume ISBN 978-85-63688-00-2, 2011.

\bibitem{JCP}
F.~Golse, F.~Hecht, O.~Pironneau, D.~Smetz, and P.-H. Tournier.
\newblock Radiative transfer for variable 3d atmospheres.
\newblock {\em J. Comp. Physics}, 475(111864):1--19, 2023.

\bibitem{FGOP3}
F.~Golse and O.~Pironneau.
\newblock Stratified radiative transfer in a fluid and numerical applications
  to earth science.
\newblock {\em SIAM Journal on Numerical Analysis}, 60(5):2963--3000, 2022.

\bibitem{papanico}
G.~Papanicolaopu L.~Ryzhik and J.B. Keller.
\newblock Transport equations for elastic and other waves in random media.
\newblock {\em Wave Motion}, 24:327--370, 1996.

\bibitem{fresnel}
O.~Lehtikangas, T.~Tarvainen, A.-D. Kim, and S.-R. Arridge.
\newblock Finite element approximation of the radiative transport equation in a
  medium with piece-wise constant refractive index.
\newblock {\em Journal of Computational Physics}, 282:345--359, 2015.

\bibitem{Liu}
L.-H. Liu.
\newblock Finite volume method for radiation heat transfer in graded index
  medium.
\newblock {\em Journal Of Thermophysics And Heat Transfer}, 20(1):59--66, Jan
  2006.

\bibitem{OP2023}
O.~Pironneau.
\newblock Numerical simulation of polarized light with rayleigh scattering in a
  stratified atmosphere.
\newblock {\em Pure and Applied Functional Analysis}, special issue dedicated
  to Luc Tartar, 2024.

\bibitem{JCP2}
O.~Pironneau and P.-H. Tournier.
\newblock Reflective conditions for radiative transfer in integral form with
  h-matrices.
\newblock {\em Journal of Computational Physics}, 495(112531):1--14, 2023.

\bibitem{POM}
G.~Pomraning.
\newblock {\em The equations of Radiation Hydrodynamics}.
\newblock Pergamon Press, NY, 1973.

\bibitem{POM2}
G.~Pomraning and B.~Ganapol.
\newblock Simplified radiative transfer for combined rayleigh and isotropic
  scattering.
\newblock {\em The Astrophysical Journal}, 498:671--688, 1998.

\bibitem{zhao}
J.M. Zhao, J.Y. Ta, and L.H. Liu.
\newblock On the derivation of vector radiative transfer equation for polarized
  radiative transport in graded index media,.
\newblock {\em Journal of Quantitative Spectroscopy and Radiative Transfer},
  113:239--250, 2012.

\end{thebibliography}

\end{document}